\documentclass[12pt,a4paper]{amsart}
\usepackage[width=0.75\paperwidth,height=0.85\paperheight,twoside=false,includehead,includefoot]{geometry}
\allowdisplaybreaks[1]

\DeclareFontEncoding{OT2}{}{}
\DeclareFontSubstitution{OT2}{cmr}{m}{n}

\DeclareFontFamily{OT2}{cmr}{\hyphenchar\font45 }
\DeclareFontShape{OT2}{cmr}{m}{n}{%
   <5><6><7><8><9>gen*wncyr%
   <10><10.95><12><14.4><17.28><20.74><24.88>wncyr10}{}

\DeclareMathAlphabet{\mathcyr}{OT2}{cmr}{m}{n}
\DeclareMathAlphabet{\mathcyb}{OT2}{cmr}{b}{n}
\SetMathAlphabet{\mathcyr}{bold}{OT2}{cmr}{b}{n}

\usepackage{amsopn}

\usepackage{amstext}
\usepackage{amsthm}
\usepackage{amssymb}

\newtheorem{thm}{Theorem}[section]
\newtheorem{lem}[thm]{Lemma}

\theoremstyle{definition}
\newtheorem{defn}[thm]{Definition}
\newtheorem{ex}[thm]{Example}
\theoremstyle{remark}
\newtheorem{rem}[thm]{Remark}
\newtheorem{que}[thm]{Question}

\newcommand{\shaub}{\mathbin{\underline{\sh}}}
\newcommand{\sht}{\mathbin{\widetilde{\sh}}}

\newcommand{\sh}{\mathbin{\mathcyr{sh}}}

\begin{document}

\title[On a generalization of the cyclic sum formula for FMZ(S)Vs]{On a generalization of the cyclic sum formula for finite multiple zeta and zeta-star values}

\author{Hideki Murahara}
\address[Hideki Murahara]{Nakamura Gakuen University Graduate School, 5-7-1, Befu, Jonan-ku,
Fukuoka, 814-0198, Japan}
\email{hmurahara@nakamura-u.ac.jp}


\begin{abstract}
 The cyclic sum formulas for multiple zeta and zeta-star values were respectively proved by Hoffman and Ohno, and Ohno and Wakabayashi. 
 Kawasaki and Oyama obtained an analogous formulas for finite multiple zeta and zeta-star values. 
 In this paper, we give a generalization of Kawasaki and Oyama's results. 
\end{abstract}

\maketitle

\section{Introduction}
\subsection{Multiple zeta(-star) values, the cyclic sum formulas, and their generalizations}
For positive integers $k_{1},\dots,k_{r}$ with $k_{r}\ge2$, 
the multiple zeta values (MZVs) and the multiple zeta-star values
(MZSVs) are defined by 
\begin{align*}
 \zeta(k_{1},\dots,k_{r})
 &:=\sum_{0<n_{1}<\cdots<n_{r}}\frac{1}{n_{1}^{k_{1}}\cdots n_{r}^{k_{r}}} \in\mathbb{R},\\
 \zeta^{\star}(k_{1},\dots,k_{r})
 &:=\sum_{0<n_{1}\le\cdots\le n_{r}}\frac{1}{n_{1}^{k_{1}}\cdots n_{r}^{k_{r}}} \in\mathbb{R}.
\end{align*}
We say that an index $(k_{1},\dots,k_{r})\in\mathbb{Z}_{\ge1}^{r}$ is admissible if $k_{r}\ge2$.
It is known that there are a lot of linear relations over $\mathbb{Q}$ among MZVs. 
The following relations obtained by Hoffman and Ohno \cite[eq.(1)]{HO03}, and Ohno and Wakabayashi \cite{OW06} are the clean-cut decompositions for the well-known sum formulas. 
\begin{thm}[Cyclic sum formula for MZ(S)Vs; Hoffman--Ohno, Ohno--Wakabayashi] \label{cycsum}
 For a non-empty index $(k_1,\dots,k_r)$ with $(k_1,\dots,k_r)\ne (\underbrace{1,\dots,1}_{r})$, we have
 \begin{align*}
  \sum_{l=1}^r \sum_{m=1}^{k_l-1} \zeta (m, k_{l+1}, \dots, k_r, k_1, \dots, k_{l-1}, k_l-m+1) 
  &=\sum_{l=1}^r \zeta (k_{l+1}, \ldots, k_r, k_1, \dots, k_{l-1}, k_l+1), \\
  \sum_{l=1}^r \sum_{m=1}^{k_l-1} \zeta^\star (m, k_{l+1}, \dots, k_r, k_1, \dots, k_{l-1}, k_l-m+1) 
  &=k \zeta^\star (k+1), 
 \end{align*}
 where we set $k:=k_1+\cdots+k_r$. 
\end{thm}
\begin{rem}
 The cyclic sum formulas for MZVs and MZSVs, i.e., the first and the second statements of Theorem \ref{cycsum}, are equivalent (see \cite[Section 4]{IKOO11} and \cite[Proposition 3.3]{TW10}). 
\end{rem}

In \cite[Theorem 2]{HMM19}, Hirose, Murakami, and the author obtained the generalization of Theorem \ref{cycsum} 
 by considering the cyclic analogue of MZVs. 
\begin{defn}[Cyclic analogue of MZVs]
Let $d$ and $r_1,\dots,r_d$ be positive integers. 
For positive integers $n_{i,1},\dots,n_{i,r_i}$ and $k_{i,1},\ldots,k_{i,r_i} \ (i=1,\dots,d)$, we write 
\begin{align*}
 \boldsymbol{k}_i
 &:=(k_{i,1},\dots,k_{i,r_i}), \\
 \boldsymbol{n}_i^{\boldsymbol{k}_i} 
 &:=n_{i,1}^{k_{i,1}} \cdots n_{i,r_i}^{k_{i,r_i}}, \\
 r
 &:=r_1+\cdots+r_d 
\end{align*}
and
\begin{align*}
 \boldsymbol{k}
 &:=[\boldsymbol{k}_{1},\dots,\boldsymbol{k}_{d}], \\
 \boldsymbol{n}^{\boldsymbol{k}} 
 &:=\boldsymbol{n}_{1}^{\boldsymbol{k}_{1}}\cdots\boldsymbol{n}_{d}^{\boldsymbol{k}_{d}}. 
\end{align*}
A multi-index $\boldsymbol{k}$ is called an admissible multi-index if 
\begin{itemize}
 \item for all $1\le i\le d$, the index $\boldsymbol{k}_i$ is admissible or equal to $(1)$, 
 \item there exists $1\le i\le d$ such that $\boldsymbol{k}_i \ne (1)$. 
\end{itemize}
Then, for an admissible multi-index $\boldsymbol{k}$, the cyclic analogue of MZVs are defined by
 \begin{align*}
  \zeta^{\mathrm{cyc}} (\boldsymbol{k}) 
  :=\sum_{S} \frac{ 1 }{ \boldsymbol{n}^{\boldsymbol{k}} },
 \end{align*}
 where
 \begin{align*}
  S&
  :=\{(n_{1,1},\dots,n_{d,r_{d}})\in\mathbb{Z}_{\ge1}^{r} 
  \mid 
  n_{1,1}<\cdots<n_{1,r_{1}},\dots, n_{d,1}<\cdots<n_{d,r_{d}}, \\
  &\qquad\qquad\qquad\qquad\qquad\qquad\,\,
   n_{1,1} \le n_{2,r_{2}}, \dots, n_{d-1,1} \le n_{d,r_{d}}, n_{d,1} \le n_{1,r_{1}} \}. 
 \end{align*}
\end{defn}
\begin{rem}
 When $d=1$, we have 
 $\zeta^{\mathrm{cyc}} ([(k_{1,1},\dots,k_{1,r_1})])=\zeta (k_{1,1},\dots,k_{1,r_1})$.
\end{rem}

We recall Hoffman's algebraic setup with a slightly different convention (see \cite{hoffman_97}).
Set $\mathfrak{H}:=\mathbb{Q}\langle x,y\rangle$.
We denote by $\mathfrak{H}^{\mathrm{cyc}}_{0}$ the subspace of $\oplus_{d=1}^{\infty}\mathfrak{H}^{\otimes d}$ 
spanned by 
\[
 \bigcup_{d=1}^{\infty}
 \{w_{1}\otimes\cdots\otimes w_{d}\in\mathfrak{H}^{\otimes d} 
 \mid w_{1},\dots,w_{d}\in y\mathfrak{H}x \cup\{y\}\  
 \text{and there exists }i\ \text{such that }w_{i}\ne y\}.
\]
For a positive integer $k$, put $z_k:=yx^{k-1}$. 
We define a $\mathbb{Q}$-linear map $Z^{\mathrm{cyc}}_{0}:\mathfrak{H}^\mathrm{cyc}\to\mathbb{R}$
by 
\[
 Z^\mathrm{cyc}(z_{k_{1,1}}\cdots z_{k_{1,r_{1}}}\otimes\cdots\otimes z_{k_{d,1}}\cdots z_{k_{d,r_{d}}})
 :=\zeta^\mathrm{cyc}([(k_{1,1},\dots,k_{1,r_{1}}),\dots,(k_{d,1},\dots,k_{d,r_{d}})]).
\]
We define the shuffle product as the $\mathbb{Q}$-bilinear product $\sh:\mathfrak{H}\times\mathfrak{H}\to\mathfrak{H}$ given by  
\begin{align*}
 1\sh w&=w\sh 1=w, \\
 wu\sh w'u'&=(w\sh w'u')u+(wu\sh w')u',
\end{align*}
where $w,w'\in\mathfrak{H}$ and $u,u'\in\{x,y\}$. 
\begin{thm}[Cyclic relation; Hirose--Murahara--Murakami] \label{cycrel}
 For $w_{1}\otimes\cdots\otimes w_{d}\in\mathfrak{H}^\mathrm{cyc}_{0}$, we have 
 \begin{align*}
  &\sum_{i=1}^{d} Z^{\mathrm{cyc}}
   (w_{1}\otimes\cdots\otimes w_{i-1}\otimes(y\shaub w_{i})\otimes w_{i+1}\otimes\cdots\otimes w_{d}) \\
  &=\sum_{i=1}^{d} Z^{\mathrm{cyc}}
   (w_{1}\otimes\cdots\otimes w_{i}\otimes y\otimes w_{i+1}\otimes\cdots\otimes w_{d}),
 \end{align*} 
 where $y\shaub u_{i}=y\sh u_{i}-yu_{i}-u_{i}y$. 
\end{thm}
\begin{rem}
 Writing $w_i=z_{k_{i,1}}\cdots z_{k_{i,r_i}}$ for $i=1,\dots,d$, we find that the case $r_1=\cdots=r_d=1$ of Theorem \ref{cycrel} gives Theorem \ref{cycsum} (for details, see \cite[Section 5.1]{HMM19}). 
\end{rem}
\begin{rem}
 Recently, Onozuka and the author \cite{MO20} gave complex variable generalization of Theorem \ref{cycrel}. 
\end{rem}

\subsection{Finite multiple zeta(-star) values and the cyclic sum formulas}
We set a $\mathbb{Q}$-algebra $\mathcal{A}$ by
\[
 \mathcal{A}:=\biggl(\prod_{p}\mathbb{Z}/p\mathbb{Z}\biggr)\,\Big/\,\biggl(\bigoplus_{p}\mathbb{Z}/p\mathbb{Z\biggr)},
\]
where $p$ runs over all primes.
For positive integers $k_{1},\dots,k_{r}$,
the finite multiple zeta values (FMZVs) and the finite multiple zeta-star
values (FMZSVs) are defined by 
\begin{align*}
 \zeta_{\mathcal{A}}(k_{1},\dots,k_{r})
 &:=\biggl(\sum_{0<n_{1}<\cdots<n_{r}<p} \frac{1}{n_{1}^{k_{1}}\cdots n_{r}^{k_{r}}}\bmod p \biggr)_{p} \in\mathcal{A}, \\
 \zeta_{\mathcal{A}}^{\star}(k_{1},\dots,k_{r})
 &:=\biggl(\sum_{0<n_{1}\le\cdots\le n_{r}<p} \frac{1}{n_{1}^{k_{1}}\cdots n_{r}^{k_{r}}}\bmod p \biggr)_{p} \in\mathcal{A}
\end{align*}
(for details of FMZVs, see \cite{Kan19,KZ20}).

Many families of relations among MZ(S)Vs have analogues for FMZ(S)Vs. 
The counterparts of Theorem \ref{cycsum} for FMZ(S)Vs were obtained by Kawasaki and Oyama \cite{KO19}.
\begin{thm}[Cyclic sum formula for FMZ(S)Vs; Kawasaki--Oyama] \label{cycsumF}
 For a non-empty index $(k_1,\dots,k_r)$ with $(k_1,\dots,k_r)\ne (\underbrace{1,\dots,1}_{r})$, we have
 \begin{align*}
  &\sum_{l=1}^{r} \sum_{m=1}^{k_l-1}
   \zeta_\mathcal{A} (m,k_{l+1},\dots,k_r,k_1,\dots, k_{l-1}, k_{l}-m+1) \\
  &=\sum_{l=1}^{r} 
   \bigl(\zeta_\mathcal{A} (k_{l+1},\dots,k_r,k_1,\dots,k_{l-1},k_{l}+1)
    +\zeta_\mathcal{A} (k_{l+1}+1,k_{l+2},\dots,k_r,k_1,\dots,k_{l}) \\
  &\qquad\quad +\zeta_\mathcal{A} (1,k_{l+1},\dots,k_r,k_1,\dots,k_{l}) \bigr), \\
  &\sum_{l=1}^{r} \sum_{m=1}^{k_l-1}
   \zeta_\mathcal{A}^\star (m,k_{l+1},\dots,k_r,k_1,\dots, k_{l-1}, k_{l}-m+1) 
   =\sum_{l=1}^{r} \zeta_\mathcal{A}^\star (1,k_{l+1},\dots,k_r,k_1,\dots,k_{l}). 
 \end{align*}
\end{thm}
\begin{rem}
 Note that the cyclic sum formulas for FMZVs and FMZSVs are equivalent (see \cite[Section 2]{MOno19} and \cite[Section 6]{HMOno19}). 
\end{rem}

\subsection{Main result}
To state our main theorem, we introduce the cyclic analogue of FMZVs (FCMZVs).  
\begin{defn}[Cyclic analogue of FMZVs]
 Let $d$ and $r_1,\dots,r_d$ be positive integers. 
 For a multi-index 
 $ \boldsymbol{k}=[\boldsymbol{k}_{1},\dots,\boldsymbol{k}_{d}]$ 
 with 
 $\boldsymbol{k}_i=(k_{i,1},\dots,k_{i,r_i}) \in\mathbb{Z}_{\ge1}^{r_i} \;(i=1,\ldots,d)$, we define
 \begin{align*}
  \zeta^{\mathrm{cyc}}_{\mathcal{A}} (\boldsymbol{k}) 
  :=\biggl( \sum_{S_p} \frac{ 1 }{ \boldsymbol{n}^{\boldsymbol{k}} } \bmod p \biggr)_{p} \in\mathcal{A},
 \end{align*}
 where we put
 $
 \boldsymbol{n}^{\boldsymbol{k}} 
 :=\boldsymbol{n}_{1}^{\boldsymbol{k}_{1}}\cdots\boldsymbol{n}_{d}^{\boldsymbol{k}_{d}}, \,
 \boldsymbol{n}_i^{\boldsymbol{k}_i} 
 :=n_{i,1}^{k_{i,1}} \cdots n_{i,r_i}^{k_{i,r_i}},
 $ 
 and
 \begin{align*}
  S_p&
  :=\{(n_{1,1},\dots,n_{d,r_{d}})\in\{ 1,\dots,p-1 \}^{r} 
  \mid 
  n_{1,1}<\cdots<n_{1,r_{1}},\dots, n_{d,1}<\cdots<n_{d,r_{d}}, \\
  &\qquad\qquad\qquad\qquad\qquad\qquad\qquad\qquad\quad
   n_{1,1} \le n_{2,r_{2}}, \dots, n_{d-1,1} \le n_{d,r_{d}}, n_{d,1} \le n_{1,r_{1}} \}. 
 \end{align*}
\end{defn}
%
  %

We denote by $\mathfrak{H}^{\mathrm{cyc}}$ the subspace of $\oplus_{d=1}^{\infty}\mathfrak{H}^{\otimes d}$ 
spanned by 
\[
 \bigcup_{d=1}^{\infty}
 \{w_{1}\otimes\cdots\otimes w_{d}\in\mathfrak{H}^{\otimes d} 
 \mid w_{1},\dots,w_{d}\in y\mathfrak{H} \}.
\]
We define a $\mathbb{Q}$-linear map $Z_{\mathcal{A}}^{\mathrm{cyc}}:\mathfrak{H}^\mathrm{cyc}\to\mathcal{A}$
by 
\[
 Z_{\mathcal{A}}^\mathrm{cyc} (z_{k_{1,1}}\cdots z_{k_{1,r_{1}}}\otimes\cdots\otimes z_{k_{d,1}}\cdots z_{k_{d,r_{d}}})
 :=\zeta_{\mathcal{A}}^\mathrm{cyc}([(k_{1,1},\dots,k_{1,r_{1}}),\dots,(k_{d,1},\dots,k_{d,r_{d}})]).
\]
\begin{thm}[Main theorem] \label{main}
 For $w_{1}\otimes\cdots\otimes w_{d}\in\mathfrak{H}^\mathrm{cyc}$, we have 
 \begin{align*}
  &\sum_{i=1}^{d} Z_{\mathcal{A}}^{\mathrm{cyc}}
   (w_{1}\otimes\cdots\otimes w_{i-1}\otimes(y\sht w_{i})\otimes w_{i+1}\otimes\cdots\otimes w_{d}) \\
  &=2\sum_{i=1}^{d} \sum_{j=1}^{r_i} 
   Z_{\mathcal{A}}^{\mathrm{cyc}}
   (w_{1}\otimes\cdots\otimes w_{i-1}\otimes 
    z_{k_{i,1}} \cdots z_{k_{i,j-1}} z_{k_{i,j}+1} z_{k_{i,j+1}} \cdots z_{k_{i,r_i}} 
    \otimes w_{i+1}\otimes\cdots\otimes w_{d}), 
 \end{align*} 
 where we put
 $w_i:=z_{k_{i,1}}\cdots z_{k_{i,r_i}}$
 and
 \begin{align*}
 &y\sht z_{k}:=
 \begin{cases}
  -y^2 &\textrm{if }k=1, \\
  0 &\textrm{if }k=2, \\
  yx(y\sh x^{k-3})x &\textrm{if }k\ge3, 
 \end{cases} \\
 &y\sht z_{k_1}\cdots z_{k_r}:=(y\sht z_{k_1})z_{k_2}\cdots z_{k_r}+\cdots+z_{k_1}\cdots z_{k_{r-1}}(y\sht z_{k_r}).
\end{align*} 
\end{thm}
\begin{ex}
 When $\boldsymbol{k}=[(1,2,3)]$, we have 
 \begin{align*}
  2\zeta_{\mathcal{A}} (1,2,4) +2\zeta_{\mathcal{A}} (1,3,3) +2\zeta_{\mathcal{A}} (2,2,3) 
  +\zeta_{\mathcal{A}} (1,1,2,3) -\zeta_{\mathcal{A}} (1,2,2,2)=0. 
 \end{align*}
 When $\boldsymbol{k}=[(1,3),(2)]$, we also have 
 \begin{align*}
  &2\zeta^{\mathrm{cyc}}_{\mathcal{A}} ([(1,3), (3)]) +2\zeta^{\mathrm{cyc}}_{\mathcal{A}} ([(1,4), (2)]) 
   +2\zeta^{\mathrm{cyc}}_{\mathcal{A}} ([(2,3), (2)]) \\
  &+\zeta^{\mathrm{cyc}}_{\mathcal{A}} ([(1,1,3), (2)]) 
   -\zeta^{\mathrm{cyc}}_{\mathcal{A}} ([(1,2,2), (2)])=0 
 \end{align*}
 and
 \begin{align*}
  &4\zeta_{\mathcal{A}}(1, 6) +2\zeta_{\mathcal{A}}(2, 5) +2\zeta_{\mathcal{A}}(3, 4) +4\zeta_{\mathcal{A}}(4, 3) 
   +\zeta_{\mathcal{A}}(1, 1, 5) +\zeta_{\mathcal{A}}(1, 2, 4) \\
  &+3\zeta_{\mathcal{A}}(1, 3, 3) -\zeta_{\mathcal{A}}(1, 4, 2) +2\zeta_{\mathcal{A}}(2, 2, 3) 
   +\zeta_{\mathcal{A}}(3, 1, 3) -\zeta_{\mathcal{A}}(3, 2, 2) \\
  &+\zeta_{\mathcal{A}}(1, 1, 2, 3) +\zeta_{\mathcal{A}}(1, 2, 1, 3) -2\zeta_{\mathcal{A}}(1, 2, 2, 2)=0.
 \end{align*}
\end{ex}
\begin{rem}
 The case $r_1=\cdots=r_d=1$ of Theorem \ref{main} implies Theorem \ref{cycsumF}. 
 In fact, by the theorem, we have
 \begin{align*}
  &\sum_{i=1}^{d} Z_{\mathcal{A}}^{\mathrm{cyc}}
   (z_{k_{1}}\otimes\cdots\otimes z_{k_{i-1}}\otimes(y\sht z_{k_{i}})\otimes z_{k_{i+1}}\otimes\cdots\otimes z_{k_{d}}) \\
  &=2\sum_{i=1}^{d} \sum_{j=1}^{r_i} 
   Z_{\mathcal{A}}^{\mathrm{cyc}}
   (z_{k_{1}}\otimes\cdots\otimes z_{k_{i-1}}\otimes 
    z_{k_{i}+1} 
    \otimes z_{k_{i+1}}\otimes\cdots\otimes z_{k_{d}}).  
 \end{align*} 
 Note that the R.H.S.\ of the above equality equals $0$ by the well-known identity of FMZVs $\zeta_{\mathcal{A}}(k)=0$.
 Since $y\sht z_k=y(y\sh x^{k-2})x -y^2x^{k-1}=\sum_{m=1}^{k-1} z_m z_{k+1-m} -y^2x^{k-1}$ and by the definition of FCMZVs, 
 we get the second statement of Theorem \ref{cycsumF}. 
\end{rem}

\section{Proof of Theorem \ref{main}}
Let
 \begin{align*}
  S_p^{(1)}&:=\{(n_{1,1},\dots,n_{d,r_{d}})\in\{ 1,\dots,p-1 \}^{r}  
   \mid 
   n_{1,1}<\cdots<n_{1,r_{1}}, \dots, n_{d,1}<\cdots<n_{d,r_{d}}, \\
  &\qquad\qquad\qquad\qquad\qquad\qquad\qquad\qquad\qquad\qquad\quad n_{1,1} \le n_{2,r_{2}}, \dots, n_{d-1,1} \le n_{d,r_{d}} \}, \\
  S_p^{(i)}&:=\{(n_{1,1},\dots,n_{d,r_{d}})\in\{ 1,\dots,p-1 \}^{r}  
   \mid 
   n_{1,1}<\cdots<n_{1,r_{1}}, \dots, n_{d,1}<\cdots<n_{d,r_{d}}, \\
  &\qquad n_{1,1} \le n_{2,r_{2}}, \dots, n_{i-2,1} \le n_{i-1,r_{i-1}}, 
   n_{i,1} \le n_{i+1,r_{i+1}}, \dots, n_{d-1,1} \le n_{d,r_{d}}, n_{d,1} \le n_{1,r_{1}} \} \\
   &\qquad\qquad\qquad\qquad\qquad\qquad\qquad\qquad\qquad\qquad\qquad\qquad\qquad\qquad\qquad\qquad\quad (i\ne1).
 \end{align*}
In the following, we understand $n_{0,j}=n_{d,j}$ and $n_{d+1,j}=n_{1,j}$.
\begin{lem} \label{lem1}
 For a positive integer $i$ with $1\le i\le d$ and $w_{1}\otimes\cdots\otimes w_{d}\in\mathfrak{H}^\mathrm{cyc}$, we have 
 \begin{align*} 
  &Z_{\mathcal{A}}^{\mathrm{cyc}}
   (w_{1}\otimes\cdots\otimes w_{i-1}\otimes(y\sht w_{i})\otimes w_{i+1}\otimes\cdots\otimes w_{d}) \\
  &=\biggl( 
   \sum_{j=1}^{r_i-1} \sum_{\substack{ S_p \\ n_{i,j}<n<n_{i,j+1} }} 
    \frac{ 1 }{ \boldsymbol{n}^{\boldsymbol{k}} } 
    \biggl( \frac{ n_{i,j} }{ n(n-n_{i,j}) } -\frac{ n_{i,j}^{k_{i,j}-1} }{ n^{k_{i,j}-1} (n-n_{i,j}) } \biggr) \\
   &\qquad +\sum_{\substack{ S_p^{(i)} \\ n_{i,r_i}<n<p \\ n_{i-1,1}\le n<p }} 
   \frac{ 1 }{ \boldsymbol{n}^{\boldsymbol{k}} } 
    \biggl( \frac{ n_{i,r_i} }{ n(n-n_{i,r_i}) } -\frac{ n_{i,r_i}^{k_{i,r_i}-1} }{ n^{k_{i,r_i}-1} (n-n_{i,r_i}) } \biggr) 
   \bmod p \biggr)_{p}, 
 \end{align*}
 where we understand $n_{i,r_i+1}=p$ for all $i$ with $1\le i\le d$. 
\end{lem}
\begin{proof}
 Let $w_i=z_{k_{i,1}}\cdots z_{k_{i,r_i}}$. 
 By definitions, we have 
 \begin{align*}
  \textrm{L.H.S.} 
  =\biggl( 
   \sum_{j=1}^{r_i-1} 
   \sum_{\substack{ S_p \\ n_{i,j}<n<n_{i,j+1} }} 
   \frac{ N_{k_{i,j}} }{ \boldsymbol{n}^{\boldsymbol{k}} } 
   +\sum_{\substack{ S_p^{(i)} \\ n_{i,r_i}<n<p \\ n_{i-1,1}\le n<p }} 
    \frac{ N_{k_{i,r_i}} }{ \boldsymbol{n}^{\boldsymbol{k}} } 
   \quad \bmod p \biggr)_{p}, 
  \end{align*}
 where
 \begin{align*}
  N_{k_{i,j}}
  =
  \begin{cases}
  \displaystyle{
  -\frac{1}{n}
  } & \textrm{if } k_{i,j}=1,
  \\  
  \;\; 0 & \textrm{if } k_{i,j}=2,
  \\  
  \displaystyle{
  \sum_{m=2}^{k_{i,j}-1} 
  \frac{ n_{i,j}^{k_{i,j}} }{ n_{i,j}^m n^{k_{i,j}-m+1} }
  } & \textrm{if } k_{i,j}\ge3
  \end{cases}
 \end{align*}
 for $j=1,\ldots,r_i$.
 Since 
 \[
  N_{1}
  =\frac{ n_{i,j} }{ n(n-n_{i,j}) } -\frac{ 1 }{ n-n_{i,j} }
 \] 
 and
 \begin{align*}
  N_{k_{i,j}} 
  &=\frac{ n_{i,j}^{k_{i,j}-1} ( 1-(n/n_{i,j})^{k_{i,j}-2} ) }{ n^{k_{i,j}-1}(n_{i,j}-n) } \\
  &=\frac{ n_{i,j} }{ n (n-n_{i,j}) } 
   -\frac{ n_{i,j}^{k_{i,j}-1}  }{ n^{k_{i,j}-1} (n-n_{i,j}) }
   \qquad (k_{i,j}\ge3),
 \end{align*} 
 we obtain the result. 
\end{proof}

For positive integers $p$ and $i$ with $1\le i\le d$, and a multi-index $\boldsymbol{k}$, let
\begin{align*}
 A_i&:=A(p,i,\boldsymbol{k}) \\
 &:=\sum_{ S_p } 
  \frac{1}{\boldsymbol{n}^{\boldsymbol{k}}} \left(
  \biggl( \frac{ 1 }{ n_{i,1}+1 }+\cdots+\frac{ 1 }{ n_{i,2}-1 } \biggr) +\cdots+
  \biggl( \frac{ 1 }{ n_{i,r_{i}-1}+1 }+\cdots+\frac{ 1 }{ n_{i,r_i}-1 } \biggr) 
  \right), \\
 B_i&:=B(p,i,\boldsymbol{k}) \\ 
 &:=\sum_{ S_p^{(i)} } 
   \displaystyle 
    \frac{1}{\boldsymbol{n}^{\boldsymbol{k}}} \biggl( 
     \biggl( \frac{1}{\max\{ 1, n_{i-1,1}-n_{i,r_i} \}}+\cdots+\frac{1}{ p-n_{i,r_i}-1 } \biggr) \\
  &\qquad\qquad\qquad\qquad\qquad\qquad\qquad\qquad -
   \biggl( \frac{1}{ \max\{ n_{i,r_i}+1, n_{i-1,1} \} } +\cdots+\frac{ 1 }{ p-1 } \biggr) 
   \biggr), \\ 
 C_i&:=C(p,i,\boldsymbol{k}) \\ 
 &:=\sum_{ S_p^{(i+1)} } 
   \displaystyle 
    \frac{1}{\boldsymbol{n}^{\boldsymbol{k}}} \biggl( 
    \biggl( \frac{1}{ \max\{ 1, n_{i,1}-n_{i+1,r_{i+1}} \} }+\cdots+\frac{1}{ n_{i,1}-1 } \biggr) \\
  &\qquad\qquad\qquad\qquad\qquad\qquad\qquad\qquad\quad\,\, +
   \biggl( 1+\cdots+\frac{1}{ \min\{ n_{i,1}-1, n_{i+1,r_{i+1}} \}} \biggr) 
   \biggr).
\end{align*}
\begin{lem} \label{lem2}
For positive integers $p$ and $i$ with $1\le i\le d$, and a multi-index $\boldsymbol{k}$, we have 
 \begin{align*} 
  \sum_{j=1}^{r_i-1} \sum_{\substack{ S_p \\ n_{i,j}<n<n_{i,j+1} }} 
   \frac{ 1 }{ \boldsymbol{n}^{\boldsymbol{k}} } 
    \biggl( \frac{ n_{i,j} }{ n(n-n_{i,j}) } -\frac{ n_{i,j}^{k_{i,j}-1} }{ n^{k_{i,j}-1} (n-n_{i,j}) } \biggr)& \\
   +\sum_{\substack{ S_p^{(i)} \\ n_{i,r_i}<n<p \\ n_{i-1,1}\le n<p }} 
   \frac{ 1 }{ \boldsymbol{n}^{\boldsymbol{k}} } 
    \biggl( \frac{ n_{i,r_i} }{ n(n-n_{i,j}) } -\frac{ n_{i,r_i}^{k_{i,r_i}-1} }{ n^{k_{i,r_i}-1} (n-n_{i,r_i}) } \biggr) 
  &=-2A_i+B_i-C_i.
 \end{align*}
\end{lem}
\begin{proof}
 Let
 \begin{align*}
  D_{i,j}^{(1)}
  &:=\sum_{\substack{ S_p \\ n_{i,j}<n<n_{i,j+1} }} 
   \frac{ 1 }{ \boldsymbol{n}^{\boldsymbol{k}} } 
   \frac{ n_{i,j} }{ n(n-n_{i,j}) }
   \quad (j\ne r_i), 
  &D_{i,r_i}^{(1)}
  &:=\sum_{\substack{ S_p^{(i)} \\ n_{i,r_i}<n<p \\ n_{i-1,1}\le n<p }} 
   \frac{ 1 }{ \boldsymbol{n}^{\boldsymbol{k}} } 
   \frac{ n_{i,r_i} }{ n(n-n_{i,r_i}) }, \\
  D_{i,j}^{(2)}
  &:=\sum_{\substack{ S_p \\ n_{i,j}<n<n_{i,j+1} }} 
   \frac{ 1 }{ \boldsymbol{n}^{\boldsymbol{k}} } 
   \frac{ n_{i,j}^{k_{i,j}-1} }{ n^{k_{i,j}-1} (n-n_{i,j}) } 
   \quad (j\ne r_i), 
  &D_{i,r_i}^{(2)}
  &:=\sum_{\substack{ S_p^{(i)} \\ n_{i,r_i}<n<p \\ n_{i-1,1}\le n<p }}
   \frac{ 1 }{ \boldsymbol{n}^{\boldsymbol{k}} } 
   \frac{ n_{i,r_i}^{k_{i,r_i}-1} }{ n^{k_{i,r_i}-1} (n-n_{i,r_i}) }.
 \end{align*}
 Then we have
 \begin{align} 
 \begin{split} \label{eq1}
  D_{i,j}^{(1)}
  &=\sum_{\substack{ S_p \\ n_{i,j}<n<n_{i,j+1} }} 
   \frac{ 1 }{ \boldsymbol{n}^{\boldsymbol{k}} } 
   \biggl( \frac{ 1 }{ n-n_{i,j} } -\frac{ 1 }{ n } \biggr) \\
  &=\sum_{S_p } 
   \frac{ 1 }{ \boldsymbol{n}^{\boldsymbol{k}} } 
   \biggl(
   \biggl( 1+\cdots+ \frac{ 1 }{ n_{i,j+1}-n_{i,j}-1 } \biggr) 
   -\biggl( \frac{ 1 }{ n_{i,j}+1 }+\cdots+ \frac{ 1 }{ n_{i,j+1}-1 } \biggr)
   \biggr)
 \end{split}
 \end{align}
 for $j=1,\dots,r_i-1$, and
 \begin{align*}
  D_{i,r_i}^{(1)}
  =\sum_{\substack{ S_p^{(i)} \\ n_{i,r_i}<n<p \\ n_{i-1,1}\le n<p }}
   \frac{ 1 }{ \boldsymbol{n}^{\boldsymbol{k}} } 
   \biggl( \frac{ 1 }{ n-n_{i,r_i} } -\frac{ 1 }{ n } \biggr) 
  =B_i.
 \end{align*}
 We also have
 \begin{align*}
  D_{i,j}^{(2)}
  %
  &=\sum_{\substack{ S_p \\ n_{i,j}<n<n_{i,j+1} }} 
   \frac{ n_{i,j}^{k_{i,j}} }{ \boldsymbol{n}^{\boldsymbol{k}} } 
   \frac{ 1 }{ n_{i,j} n^{k_{i,j}-1} (n-n_{i,j}) } \\
  &=\sum_{\substack{ S_p \\ n_{i,j-1}<n<n_{i,j} }} 
   \frac{ n_{i,j}^{k_{i,j}} }{ \boldsymbol{n}^{\boldsymbol{k}} } 
   \frac{ 1 }{ n_{i,j}^{k_{i,j}-1} n (n_{i,j}-n) }  
 \end{align*}
 for $j=2,\dots,r_i-1$. 
 In the above equality, we interchanged $n$ and $n_{i,j}$. 
 Then we have
 \begin{align}
 \begin{split} \label{eq2}
  D_{i,j}^{(2)}
  &=\sum_{\substack{ S_p \\ n_{i,j-1}<n<n_{i,j} }}
   \frac{ 1 }{ \boldsymbol{n}^{\boldsymbol{k}} } 
   \biggl( \frac{ 1 }{ n_{i,j}-n } +\frac{ 1 }{ n } \biggr) \\
  &=\sum_{S_p } 
   \frac{ 1 }{ \boldsymbol{n}^{\boldsymbol{k}} } 
   \biggl(
   \biggl( 1+\cdots+ \frac{ 1 }{ n_{i,j}-n_{i,j-1}-1 } \biggr) 
   +\biggl( \frac{ 1 }{ n_{i,j-1}+1 }+\cdots+ \frac{ 1 }{ n_{i,j}-1 } \biggr)
   \biggr)
 \end{split}
 \end{align}
 for $j=2,\dots,r_i-1$. 
 Similarly, we have
 \begin{align*}
  D_{i,1}^{(2)}
  &=\sum_{\substack{ S_p^{(i+1)} \\ 1\le n<n_{i,1} \\ 1\le n\le n_{i+1,r_{i+1}} }}
   \frac{ 1 }{ \boldsymbol{n}^{\boldsymbol{k}} } 
   \biggl( \frac{ 1 }{ n_{i,1}-n } +\frac{ 1 }{ n } \biggr) 
  =C_i
 \end{align*}
 and
 \begin{align}
 \begin{split} \label{eq3}
  D_{i,r_i}^{(2)}
  &=\sum_{S_p }
   \frac{ 1 }{ \boldsymbol{n}^{\boldsymbol{k}} } 
   \biggl(
    \biggl( 1+\cdots+ \frac{ 1 }{ n_{i,r_i}-n_{i,r_i-1}-1 } \biggr) 
    +\biggl( \frac{ 1 }{ n_{i,r_i-1}+1 }+\cdots+ \frac{ 1 }{ n_{i,r_i}-1 } \biggr)
   \biggr). 
 \end{split}
 \end{align}
 From \eqref{eq1}, \eqref{eq2}, and \eqref{eq3}, we have
 \begin{align*}
  D_{i,j}^{(1)}-D_{i,j+1}^{(2)}
  =-2\sum_{S_p } 
   \frac{ 1 }{ \boldsymbol{n}^{\boldsymbol{k}} } 
   \biggl( \frac{ 1 }{ n_{i,j}+1 }+\cdots+ \frac{ 1 }{ n_{i,j+1}-1 } \biggr)
 \end{align*}
 for $j=1,\dots,r_i-1$.
 Hence we find the result. 
\end{proof}

\begin{proof}[Proof of Theorem \ref{main}]
 Note that
\begin{align*}
 B_i
 &=\sum_{ S_p } 
   \frac{1}{\boldsymbol{n}^{\boldsymbol{k}}} \biggl( 
   \biggl( 1+\cdots+\frac{1}{ p-n_{i,r_i}-1 } \biggr)
   -\biggl( \frac{1}{ n_{i,r_i}+1 } +\cdots+\frac{ 1 }{ p-1 } \biggr) 
   \biggr), \\
  &\quad +\sum_{\substack{ S_p^{(i)} \\ n_{i,r_i}<n_{i-1,1} }}  
    \frac{1}{\boldsymbol{n}^{\boldsymbol{k}}} \biggl( 
     \biggl( \frac{1}{ n_{i-1,1}-n_{i,r_i} }+\cdots+\frac{1}{ p-n_{i,r_i}-1 } \biggr) 
   -\biggl( \frac{1}{ n_{i-1,1} } +\cdots+\frac{ 1 }{ p-1 } \biggr) 
   \biggr)
 \end{align*}
 and
 \begin{align*}  
  C_i 
  &=2\sum_{ S_p } 
    \frac{1}{\boldsymbol{n}^{\boldsymbol{k}}}  
    \biggl( 1+\cdots+\frac{1}{ n_{i,1}-1 } \biggr) \\
  &\quad +\sum_{\substack{ S_p^{(i+1)} \\ n_{i+1,r_{i+1}}<n_{i,1} }}
    \frac{1}{\boldsymbol{n}^{\boldsymbol{k}}} \biggl( 
    \biggl( \frac{1}{ n_{i,1}-n_{i+1,r_{i+1}} }+\cdots+\frac{1}{ n_{i,1}-1 } \biggr) 
   +\biggl( 1+\cdots+\frac{1}{ n_{i+1,r_{i+1}} } \biggr) 
   \biggr).
 \end{align*}
 Since
 \begin{align*}
  \biggl( \frac{ 1 }{ a }+\cdots+\frac{ 1 }{ p-a } \bmod{p} \biggr)_p =0,
 \end{align*}
 we have
 \begin{align*}
  &\biggl( \frac{1}{ n_{i-1,1}-n_{i,r_i} }+\cdots+\frac{1}{ p-n_{i,r_i}-1 } \bmod{p} \biggr)_p \\
  &=\biggl( -\biggl( \frac{1}{ n_{i,r_i}+1 }+\cdots+\frac{1}{ n_{i-1,1}-n_{i,r_i}-1 } \biggr) \bmod{p} \biggr)_p. 
 \end{align*}
 By Lemmas \ref{lem1} and \ref{lem2}, we have
 \begin{align*}
  &Z_{\mathcal{A}}^{\mathrm{cyc}}
   (w_{1}\otimes\cdots\otimes w_{i-1}\otimes(y\sht w_{i})\otimes w_{i+1}\otimes\cdots\otimes w_{d}) \\
  &=(-2A_i+B_i-C_i \bmod{p} )_p \\
  &=\biggl( -2\sum_{ S_p } 
   \frac{1}{\boldsymbol{n}^{\boldsymbol{k}}} 
   \biggl( \sum_{m=1}^{p-1} \frac{ 1 }{ m } -\biggl( \frac{1}{ n_{i,1} } +\cdots+\frac{ 1 }{ n_{i,r_i} } \biggr) \biggr) \\
  &\quad -\sum_{\substack{ S_p^{(i)} \\ n_{i,r_i}<n_{i-1,1} }}  
   \frac{1}{\boldsymbol{n}^{\boldsymbol{k}}} 
   \biggl( 
    \biggl( \frac{1}{ n_{i,r_i}+1 }+\cdots+\frac{1}{ n_{i-1,1}-n_{i,r_i}-1 } \biggr) 
    +\biggl( \frac{1}{ n_{i-1,1} } +\cdots+\frac{ 1 }{ p-1 } \biggr) 
   \biggr) \\
  &\quad -\sum_{\substack{ S_p^{(i+1)} \\ n_{i+1,r_{i+1}}<n_{i,1} }}
    \frac{1}{\boldsymbol{n}^{\boldsymbol{k}}} \biggl( 
    \biggl( \frac{1}{ n_{i,1}-n_{i+1,r_{i+1}} }+\cdots+\frac{1}{ n_{i,1}-1 } \biggr) 
   +\biggl( 1+\cdots+\frac{1}{ n_{i+1,r_{i+1}} } \biggr) 
   \biggr) 
  \bmod{p}\biggr)_p.
 \end{align*}
 Then we have  
 \begin{align*}
  &\sum_{i=1}^{d} Z_{\mathcal{A}}^{\mathrm{cyc}}
   (w_{1}\otimes\cdots\otimes w_{i-1}\otimes(y\sht w_{i})\otimes w_{i+1}\otimes\cdots\otimes w_{d}) \\
  &=\biggl( 2\sum_{i=1}^{d} \sum_{j=1}^{r_i} \sum_{ S_p } 
   \frac{ 1 }{ \boldsymbol{n}^{\boldsymbol{k}} n_{i,j} } 
   \bmod{p} \biggr)_p.
 \end{align*}
 This finishes the proof. 
\end{proof}

\begin{que}
 It is known that the cyclic relation (Theorem \ref{cycrel}) includes the derivation relation for MZVs 
 obtained by Ihara, Kaneko, and Zagier  \cite[Theorem 3]{IKZ06}). 
 However, Theorem \ref{main} does not include the derivation relation for FMZVs proved by the author \cite{Mur17}.
 Are there any generalizations of Theorem \ref{main} that include this relation? 
 Let $R_x(w):=wx$ for $w\in\mathfrak{H}$. 
 Then, by Theorem \ref{cycrel} and the arguments in \cite[the last part of Section 2]{Mur17} and \cite[Section 5.3]{HMM19}, 
 it can be proved 
 \begin{align*}
  Z^{\mathrm{cyc}}_{\mathcal{A}}
   (R_x^{-1} ((y\shaub w) \otimes \underbrace{y\otimes\cdots\otimes y}_{m})) 
  =(m+1) Z^{\mathrm{cyc}}_{\mathcal{A}}
   (R_x^{-1} (w \otimes \underbrace{y\otimes\cdots\otimes y}_{m+1}))
 \end{align*} 
 for non-negative integer $m$ and $w\in y\mathfrak{H}x$, which is essentially equivalent to the derivation relation for FMZVs.  
\end{que}


\end{document}